\documentclass[12pt]{amsart}
\usepackage[dvips]{graphicx}
\usepackage{amssymb,amsmath,amsthm,cases}

\def\NZQ{\mathbb}               % the font for N,Z,Q,R,C
\def\NN{{\NZQ N}}

\def\frk{\mathfrak}               % font for "Fraktur"

\def\Phi{{\frk N}}
\def\Pc{{\mathcal P}}

\def\I{{\mathcal I}}

\def\opn#1#2{\def#1{\operatorname{#2}}} % to make operators
\opn\chara{char} \opn\length{\ell} \opn\pd{pd} \opn\rk{rk}
\opn\projdim{proj\,dim} \opn\injdim{inj\,dim} \opn\rank{rank}
\opn\depth{depth} \opn\grade{grade} \opn\height{height}
\opn\size{size}
\opn\embdim{emb\,dim} \opn\codim{codim}

\opn\Tr{Tr} \opn\bigrank{big\,rank}
\opn\superheight{superheight}\opn\lcm{lcm}
\opn\trdeg{tr\,deg}%\emph{
\opn\reg{reg} \opn\lreg{lreg} \opn\ini{in} \opn\lpd{lpd}
\opn\size{size}\opn{\mult}{mult}
\opn{\Cl}{Cl}
\opn\div{div} \opn\Div{Div} \opn\cl{cl} \opn\Cl{Cl}
\opn\Spec{Spec} \opn\Supp{Supp} \opn\supp{supp} \opn\Sing{Sing}
\opn\Ass{Ass} \opn\Min{Min} \opn\cl{cl}
\opn\Ann{Ann} \opn\Rad{Rad} \opn\Soc{Soc}
\opn\Syz{Syz} \opn\Im{Im} \opn\Ker{Ker} \opn\Coker{Coker}
\opn\Am{Am} \opn\Hom{Hom} \opn\Tor{Tor} \opn\Ext{Ext}
\opn\End{End} \opn\Aut{Aut} \opn\id{id} \opn\ini{in}

\opn\nat{nat}
\opn\pff{pf}%   \pf exists already
\opn\Pf{Pf} \opn\GL{GL} \opn\SL{SL} \opn\mod{mod} \opn\ord{ord}
\opn\Gin{Gin}
\opn\Hilb{Hilb}\opn\adeg{adeg}\opn\std{std}\opn\ip{infpt}
\opn\Pol{Pol}
\opn\sat{sat}
\opn\Var{Var}
\opn\Gen{Gen}
\opn\lex{lex}
\opn\div{div}

\opn\aff{aff} \opn\con{conv} \opn\relint{relint} \opn\st{st}
\opn\lk{lk} \opn\cn{cn} \opn\core{core} \opn\vol{vol}
\opn\link{link} \opn\star{star}
\opn\gr{gr}

\def\Ic{{\mathcal I}}
\def\Jc{{\mathcal J}}

\def\Cc{{\mathcal C}}

\def\Qc{{\mathcal Q}}
\def\Rc{{\mathcal R}}

\def\pot#1#2{#1[\kern-0.28ex[#2]\kern-0.28ex]}

\opn\dirlim{\underrightarrow{\lim}}
\opn\inivlim{\underleftarrow{\lim}}
\def\Implies{\ifmmode\Longrightarrow \else
        \unskip${}\Longrightarrow{}$\ignorespaces\fi}
\def\implies{\ifmmode\Rightarrow \else
        \unskip${}\Rightarrow{}$\ignorespaces\fi}
\def\iff{\ifmmode\Longleftrightarrow \else
        \unskip${}\Longleftrightarrow{}$\ignorespaces\fi}

\let\:=\colon

\def\NZQ{\mathbb}        % the font for N,Z,Q,R,C
\def\NN{{\NZQ N}}

\def\frk{\mathfrak}        % font for "Fraktur"

\def\Phi{{\frk N}}

\def \P{{\mathcal P}}

\newtheorem{Theorem}{Theorem}[section]
\newtheorem{Lemma}[Theorem]{Lemma}

\newtheorem{Proposition}[Theorem]{Proposition}

\theoremstyle{definition}

\begin{document}
\title{Toric representations of algebras defined by certain nonsimple polyominoes}
\author {Akihiro Shikama}

\thanks{}


\address{Akihiro Shikama, Department of Pure and Applied Mathematics, Graduate School of Information Science and Technology,
Osaka University, Suita, Osaka 560-0871, Japan}
\email{a-shikama@cr.math.sci.osaka-u.ac.jp}

\begin{abstract}
%A classification problem of primeness of polyomino ideals have been discussed in many papers.
In this paper we give a toric representation of  the associated ring of a polyomino which is obtained by removing a convex polyomino from its ambient rectangle.
\end{abstract}

\maketitle
\section*{Introduction}
Polyominoes are two dimensional objects which are obtained by joining squares of equal sizes edge to edge.
They are originally rooted in recreational mathematics and combinatorics. For example, they have been studied in the tiling problems of the plane.
In combinatorial commutative algebra, polyominoes are first introduced in \cite{Q} by assigning each polyomino the ideal of inner 2-minors or the {\it polyomino ideal}.
The study of ideal of $t$-minors of an  $m \times n$
matrix is a classical subject in commutative algebra.
The class of polyomino ideals widely generalizes the class of ideals of 2-minors of $m \times n$ matrix as well as the ideals of inner 2-minors attached to a one or two sided ladder.

Let $\P$ be a polyomino and $K$ be a field.
We denote by $I_\P$, the polyomino ideal attached to $\P$,
in a suitable polynomial ring over $K$.
%The residue class ring defined by $I_\P$ denoted by $K[\P]$.
It is natural to investigate the algebraic properties of $I_\P$  %$K[\P]$
depending on shape of $\P$.
The classes of polyominoes whose polyomino ideal is 
prime have been discussed in many papers, including \cite{HQS,HM,HQ,Q} .
%In \cite{Q} it is shown that polyomino ideal of convex polyominoes are prime. 
%Later in \cite{HQS} the authors introduced the class of em balanced
%polyominoes and proved that polyomino ideal of balanced polyominoes are prime.
%Later in \cite{HM} it is shown that simple polyominoes are balanced.

%On the other hand, it is known that a binomial prime ideal is a %defining ideal of some  toric ring \cite{St}.
The most outstanding result in these studies of polyomino ideals is
 given in  \cite{QSS}.
 % independently from \cite{HM}, 
 It is proved that the polyomino ideals of simple polyominoes are prime by identifying their quotient rings with toric rings of the edge rings of graphs.

Recently in \cite{HQ} it is shown that the polyomino ideal of the  nonsimple polyomino which is obtained by removing a convex polyomino from its ambient rectangle is prime
by using a localization argument.
%In \cite{S}, the identification for polyominoes which are ``rectangle minus rectangle'' was given.
In the present paper, we give a toric representation of the quotient rings of the polyomino ideals of
this class of nonsimple polyominoes.
\section{Definition and known results}
We recall some definitions and notation from \cite{Q}.
Given $a = (i,j)$ and $b = (k,l)$ in $\NN^2$, we write $a \le b$ if $i\le k$ and $j\le l$. 
We say $a$ and $b$ are in {\em horizontal (or vertical) position}
if $j=l$ (or $i = k)$.
The set $[a,b] = \{c \in \NN^2 | a \le c \le b\}$ is  called an {\em interval}.
If $i < k $ and $j < l$ then the vertices $a$ and $b$ are called {\em diagonal corners}
and $(i,l)$ and $(k,j)$ are called {\em anti-diagonal corners} of $[a,b]$. 
The interval of the  form $C = [a,a+(1,1)]$ is called a {\em cell}.
The elements $a, a+(1,0),a+(0,1),a+(1,1)$ are called vertices of $C$. 
We denote the set of vertices of $C$ by $V(C)$.
The sets $\{a,a+(1,0)\},\{a,a+(0,1)\}, \{a+(1,0),a+(1,1)\}$ and $\{a+(0,1),a+(1,1)\}$ are called the {\em edges} of $C$. We denote the set of edges of $C$ by $E(C)$.

Let $\P$ be a finite collection of cells of $\NN^2$. The vertex set of $\Pc$ is denoted by $V(\P) = \cup_{C\in \P} V(C)$. The edge set of $\Pc$ is denoted by $E(\P ) = \cup _{C\in \P}E(C)$.
Let $C$ and $D$ be two cells of $\P$. Then $C$ and $D$ are said to be {\em connected} if there exists a sequence of cells $\Cc :C=C_1 ,\ldots,C_m=D$ such that $C_i \cap C_{i+1}$ is an edge of $C_i$ for $ i = 1, \ldots, m-1$.
If in addition, $C_i \neq C_j$ for all $i \neq j$, then $\Cc$ is called a {\em path} from $C$ to $D$.
The collection of cells $\Pc$ is called a {\em polyomino}
if any two cells of $\Pc$ are connected. For example,
Figure \ref{polyomino} shows a polyomino.

\begin{figure}[h]
\includegraphics[width= 4cm]{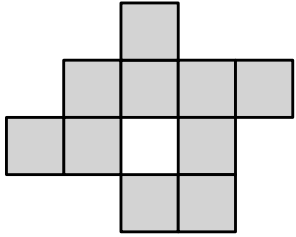}
\caption{a polyomino}\label{polyomino}
\end{figure}

Now we recall the definition of polyomino ideals from \cite{Q}.
Let $\Pc$ be a polyomino, and let $K$ be a field.
Let $S$ be the polynomial ring over $K$ with variables $x_{ij}$ with 
$(i,j) \in V(\P)$. A binomial $x_{ij}x_{kl}-x_{il}x_{kj}$ is called 
an {\em inner minor} of $\P$ if all the cells $[(r,s),(r+1,s+1)]$ with
$i\le r\le k-1$ and $j\le s \le l-1$ belong to $\Pc$.
In that case the interval $[(i,j),(k,l)]$ is called an {\em inner interval} of $\Pc$.
The ideal $I_\Pc\subset S$ generated by all inner minors of $\Pc$ is called the {\em polyomino ideal} of $\Pc$.
An interval $[a,b]$ with $a = (i,j)$ and $b=(k,l)$ is called a
horizontal edge interval of $\Pc$ if $j=l$ and the sets $\{(r,j),(r+1,j)\} \in E(\P)$ for $r = i,\ldots , k-1$.
Similarly, one defines {\em vertical edge interval}.

A polyomino $\Pc$ is called {\em simple} if 
for any two cells $C,D$ not belonging to $\P$, there exists 
a sequence of cells $C = C_1,\ldots, C_m=D$ such that each $C_i \notin \P$ and $C_i\cap C_{i+1}$ is an edge of $C_i$ for $i = 1,\ldots m-1.$
Roughly speaking, a simple polyomino is a polyomino with no ``hole''.

A polyomino $\Pc$ is called {\em  row convex} if any two cells $C = [(i,j),(i+1, j+1)],D=[(k,l),(k+1,l+1)]$ of $\P$ with $i<k$ with $j= l$ all cells $[(l,j),(l+1,j+1)] \in \P$ for $ i\le l\le k $.
Similarly, one defines {\em column convex} polyominoes.
A polyomino is called {\em convex} if it is both row and column convex.

For the polyomino ideals, the following classification of primeness is known.
\begin{Proposition}
Let $\Pc$ be a polyomino with one of the following condition.
Then $I_\P$ is a prime ideal.
\begin{enumerate}
\item $\P$ is a one sided ladder \cite{Na}.
\item $\P$ is a two sided ladder \cite{Co}.
\item $\P$ is row or column convex \cite{Q}.
\item $\P$ is balanced \cite{HQS}.
\item $\P$ is simple \cite{HM,QSS}.
\item $\P$ is obtained by removing a rectangle from its ambient rectangle \cite{S}.
\item $\P$ is obtained by removing a convex polyomino from its ambient rectangle \cite{HQ}.
\end{enumerate}
\end{Proposition} 
 Note that the first four classes of polyominoes are simple.
Recall from \cite{Na} that a simple polyomino is called a {\em one-sided ladder} 
if it is of the following type:
\begin{figure}[h]
\includegraphics[width = 4cm]{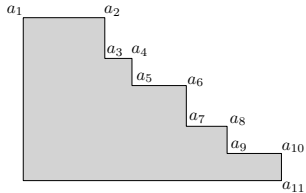}
\caption{a one-sided ladder}\label{oneside}
\end{figure}

%The vertices indicated $a_1, a_3,\ldots, a_{11}$ are called {\em inside corners} and 
%the vertices indicated $a_2,a_4,\ldots,a_{10}$ are called {\em outside corners}.
The sequence of vertices $ a_1,\ldots ,a_s$ of the corners of one-sided ladder $\P$ other than the opposite corner of the ladder
is  called the {\em  defining sequence} of $\P$ if each $a_i$ and $a_{i+1}$ are in the horizontal or vertical position.
For example, the sequence $a_1,a_2,\ldots,a_{11}$ in Figure \ref{oneside} is the defining sequence of this one-sided ladder.

 It is also known that their exist nonsimple polyominoes whose polyomino ideals are not prime.
Figure \ref{notprime} is one of such examples given in \cite{QSS}.

\begin{figure}[h]
\includegraphics[width= 4cm]{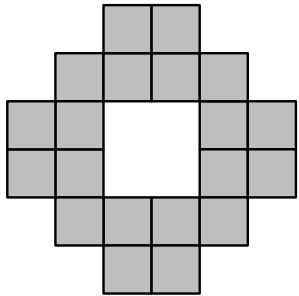}
\caption{a polyomino with non-prime polyomino ideal}\label{notprime}
\end{figure}
 
\section{The main result}
The aim of this paper is to give a toric representation of the associated ring of  a polyomino which is
obtained by removing a convex polyomino from its ambient rectangle.
In order to prove the main theorem, we give some properties of convex polyominoes.

\begin{Lemma}\label{idk}
Let $\P$ be a convex polyomino and let $\Ic$ be the unique minimal interval such that $ \P \subset \Ic$. Then,
\begin{enumerate}
\item[(a)] $\Ic \setminus \P$ consists of at most $4$ connected components;
\item[(b)] each connected component of $\Ic \setminus \P$ contains exactly one corner vertex of $\Ic$;
\item[(c)]  each connected component of $\Ic \setminus \P$ is a one-sided ladder.
\end{enumerate}
\end{Lemma}
%\begin{proof}
%(a) and (b) follows since $\P$ is convex and $\Ic$ is minimal.
%(c) Suppose $Q$ is connected component to which the left lower corner of $
%\end{proof}

Let $\P$ be  a convex polyomino.
A vertex of $\P$ is called a {\em outside corner} if it belongs to exactly one cell of $\P$.
On the other hand, a vertex of $\P$ is called a {\em inside corner} if it belongs to three cells of $\P$.
A vertex is called an {\em interior vertex} if it belongs to four cells of $\P$.
The {\em boundary vertices} are the vertices which are not interior vertices.
A cell of $\P$ is called an {\em interior cell} if all of its 4 vertices are interior vertices.
A cell of $\P$ is called an {\em boundary cell} if it is not an interior cell.
We denote the set of boundary vertices of $\P$ by $\partial \Pc$.

To each interval $[a,b]$, we attach a polyomino $\P_{[a,b]}$ in the obvious way. Such polyomino is called 
{\em rectangle}.
Hereafter, let $\P$ be a polyomino which is obtained by removing a convex polyomino
$\Qc$ from its  ambient rectangle $\P_{[a,b]}$.
We assume  $\partial \P_{[a,b]} \cap  \partial \Qc = \emptyset$,
otherwise, $\P$ is a simple polyomino and its toric representation is well studied in \cite{QSS}. Also, we assume that $a=(1,1)$ and $b = (m,n)$. 

%A vertex of $v \in \P$ is called a {\em corner} if $ v \in \Pc \cap \Qc$ and

%A corner is called an {\em inner corner} if 

%A corner is called an {\em outside corner} if it is not an inner corner.

We define two types of intervals of $\P$ as follows:

\begin{enumerate}
%\item[(i)] For each inside corner $a$ of $\partial \Qc$, $\Ic_a$ is the maximal interval in $\P$ to which $a $ belongs.
\item[(i)] For the lowest corner $e$  among all most left outside corners of $\Qc$, let
%For each outside corner $a \in \partial \Qc$, 
$\Ic_e = [a,e]$.
\item[(ii)]
The maximal vertical or horizontal intervals $\Ic$ of $\P$.
\end{enumerate}

For example,
for the given polyomino,
the intervals of types (i) and (ii) 
are displayed in Figures 4, 5.

\begin{figure}
%\begin{minipage}{0.45\hsize}
%\begin{center}
%\includegraphics[width  =6 cm]{original}
%\caption{a polyomino}
%\end{center}
%\end{minipage}
%\begin{minipage}{0.45\hsize}
%\begin{center}
%\includegraphics[width  =6 cm]{inner}
%\caption{(i) inside corners}
%\end{center}
%\end{minipage}
\begin{minipage}{0.45\hsize}
\begin{center}
\includegraphics[width  =6 cm]{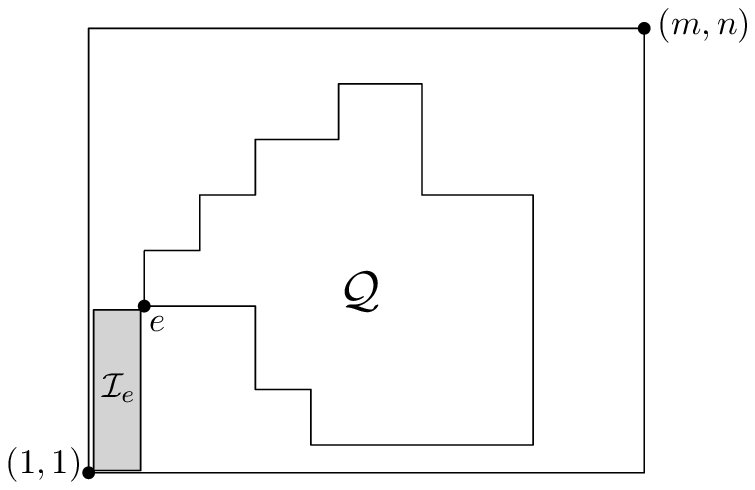}
\caption{(i) interval ${\mathcal I}_e$}\label{ie}
\end{center}
\end{minipage}
\begin{minipage}{0.45\hsize}
\begin{center}
\includegraphics[width  =6 cm]{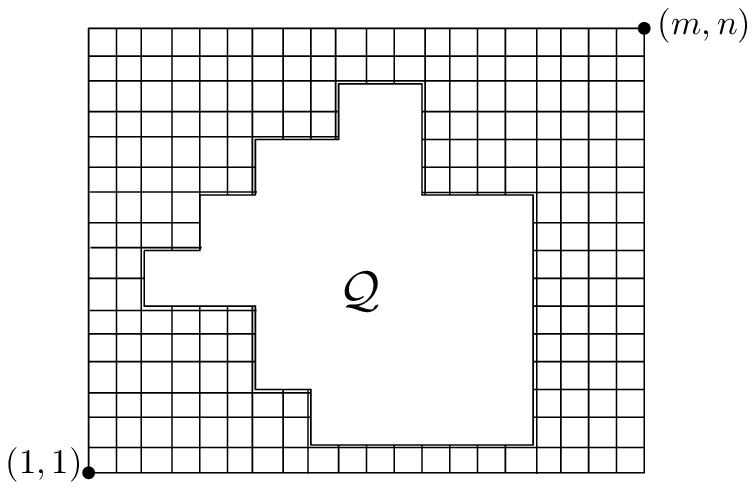}
\caption{(ii) edge intervals}
\end{center}
\end{minipage}
\label{label}
\end{figure}
We denote 
 the set of intervals of types (i) and (ii)  by $\Lambda$.
We define a map $\alpha: V(\P) \rightarrow K[\{u_\Ic\}_{\Ic \in \Lambda}]$
by $\displaystyle v \mapsto \prod_{\substack{ v \in \Ic\\ \Ic \in \Lambda}}u_\Ic$.
Now we define the toric ring and the toric ideal.
The {\em toric ring} denoted by $T$ is defined as
\[
T = K[\alpha(v) \mid v  \in V(\P)]\subset K[\{u_\Ic\}_{\Ic \in \Lambda}].
\]
Let $\varphi : S \rightarrow T$ be the surjective ring homomorphism with the setting $\varphi(x_{ij}) = \alpha((i,j))$.
The {\em toric ideal} $J_\P$ is the kernel of $\varphi$.
We claim that $J_\P = I_\P$.
In order to prove this, we will repeatedly use the next lemma.

For any binomial $f = f^+ - f^-\in J_\P$, we let $V_+$  be the set of vertices $v$ such that $x_v$ appear in $f^+$. Similarly one defines $V_-$.
A binomial $f$ in a binomial ideal $I \subset S$ is said to be {\em redundant}
if it can be expressed as a linear combination of binomials in $I$ of lower degree.
A binomial is said to be {\em irredundant} if it is not redundant.
\begin{Lemma}\label{reduction}
Let $f = f^+ - f^-$ be a binomial of degree $\ge 3$ belonging to  $J_\P$.
If there exist three vertices $p,q\in V_+$ and $r\in V_-$ such that $p,q$ are diagonal (resp. anti-diagonal) corners of an inner interval and $r$ is one of anti-diagonal (resp. diagonal) corners of the inner interval, then $f$ is redundant in $J_\P$.
\end{Lemma}
\begin{proof}
Let $s$ be the other corner of the interval determined by $p,q$ and $r$.
Then
\begin{eqnarray*}
f & = & f^+-f^-\\
& = & x_px_q\frac{f^+}{x_px_q}- f^-\\
&=& (x_px_q-x_rx_s)\frac{f^+}{x_px_q} + x_rx_s\frac{f^+}{x_px_q} -x_r\frac{f^-}{x_r}\\
&=&(x_px_q-x_rx_s)\frac{f^+}{x_px_q} +x_r\left( x_s\frac{f^+}{x_px_q} - \frac{f^-}{x_r}\right).
\end{eqnarray*}
Since $x_px_q - x_rx_s$ is an inner minor of $\P$ and since $J_\P$ is a toric ideal, we have the desired conclusion.
\end{proof}

\begin{Theorem}
Let $\P = \P_{[(1,1),(m,n)]} \setminus \Qc$ be a polyomino where $\Qc \subset \P_{[(1,1),(m,n)]}$ is a convex polyomino. Then $I_\P = J_\P$.
\end{Theorem}
\begin{proof}
%It is easy to see that $I_\P \subset J_\P$.
First we show $I_\P \subset J_\P$. Let $x_px_q-x_rx_s$ be an inner minor belonging to $I_\P$. 
Assume that $p$ is the lower left corner, $q$ is the upper right corner,
$r$ is the lower right corner and $s$ is the upper left corner of $\P$.
Since $[p,q]$ is an inner interval, it is clear that $p$ and $r$, and $q$ and $s$ belong to the same maximal horizontal intervals. 
It is also clear that $p$ and $s$, $q$ and $r$ belong to the same maximal vertical intervals. 
To show $f = f^+-f^-= x_px_q-x_rx_s \in J_\P$, it suffices to show that the number of vertices in $V_+= \{p,q\}$ belong to $\I_e$ is equal to the number of vertices in $V_- = \{r,s\}$ belong to $\I_e$.
If $p \notin \I_e$, we see that $q,r,s \notin I_e$ and we are done in this case. 
Suppose that $p \in \I_e$. Since $[p,q]$ is an inner interval, if 
$r \in \I_e$, then we have either both $q$ and $s$ belong to $\I_e$, or both 
$q$ and $s$ do not belong to $\I_e$. In these cases, we see that $x_px_q-x_r x_s \in J_\P$.
Similarly, if $r \notin \I_e$, then the only possibility is that 
$s \in \I_e$ and $q \notin \I_e$.
Thus, we have $I_\P \subset J_\P$. 
\medskip

Next, in order to prove $J_\P \subset I_\P$, it suffices to show that every binomial of degree 2
in $J_\P$ belongs to $I_\P$ and that every irredundant binomial in $J_\P$ is of degree 2.
First we show that every binomial $f \in J_\P$ of degree 2 belongs to $I_\P$.
%By the definition of $T$, we see that every binomial belongs to 
%$J_\P$ is homogeneous.
Suppose that $f = x_px_q-x_rx_s \in J_\P$ is a binomial such that $\{p,q\} \neq \{r,s\}$.

Since $\varphi(x_px_q) = \varphi(x_rx_s)$, we may assume that $[p,q]$ is an interval which has
$r$ and $s$ as its anti-diagonal corners. Assume that the pair $p$ and $r$ and the pair $s$ and $q$ belong to  the same horizontal edge interval. Then we see that the pairs $p$ and $s$ and the pair $r$ and $q$ belong to the same vertical edge interval.
If $[p,q]$ is a inner minor of $\P$ then we are done.
Suppose that $[p,q]$ is not an inner interval.
Then %Since $a,b,c$ and $d$ are vertices of $\P$, 
we have either $\Qc \subset \P_{[p,q]}$ or $\Qc \not\subset \P_{[p,q]}$ and $\Qc \cap \P_{[p,q]} \neq \emptyset$.
%The case $\Qc \subset [a,b]$ does not occur by  Lemma \ref{idk} and the definition of $\P$.
Suppose that $\Qc \subset \P_{[p,q]}$.
We see that $p \in \Ic_e$ and $q,r,s \notin \Ic_e$, where $\Ic_e$ is the interval given in Figure~\ref{ie}. Then, we have $u_{\Ic_e} | \varphi(x_p)$ and $u_{\Ic_e} | \varphi(x_rx_s)$, which contradicts to $x_px_q-x_rx_s \in J_\P$.
Hence this case is not possible.

Suppose that $\Qc \not\subset \P_{[p,q]}$ and $[p,q]$ is not an inner interval of $\P$.
We see that at least one of 
$[p,r]$, $[p,s]$, $[s,q]$ and $[r,q]$ is not an edge interval in $\P$.
Say $[p,r]$ is not an edge interval in $\P$.

\begin{figure}[h]
\includegraphics[width = 4cm]{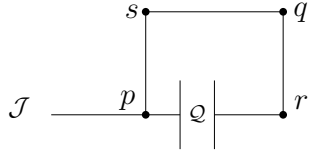}
\caption{the maximal interval}\label{int}
\end{figure}
Suppose that $\Jc \in \Lambda$ is the maximal horizontal edge interval  to which $p$ belongs. Since $x_px_q-x_rx_s \in J_\P$, we see that $u_\Jc | \varphi(x_p)$ and hence $u_\Jc | \varphi(x_rx_s)$.   
This is a contradiction to the fact that neither $r$ nor $s$ belongs to $\Jc$ (see Figure \ref{int}). 
Hence this case is not possible. Thus, every binomial $f \in J_\P$ of degree 2 belongs to $I_\P$.

\medskip

Now we show that every binomial $ f \in J_\P$ with $\deg f \ge 3$ is redundant.
Suppose that $f = f^+ -f^-$ is an irredundant binomial with $\deg f \ge 3$.

First we show that there does not  exist any vertex $v \in V_+ \cup V_-$ such that $v \in \I_e$, where $\I_e$ is the interval shown in Figure \ref{ie}.
To show this, on contrary suppose that there exist $v_1 \in V_+\cap \I_e$.
Since $\varphi(f^+) = \varphi(f^-)$, we have a vertex $v'_1 \in V_- $ such that $v'_1 \in \I_e$.
Also, we have a vertex $v'_2$ such that $v_1$ and $v'_2$ belong to the same maximal vertical  edge interval.
We see that there exist a vertex $v_2 \in V_+$ such that $v_2 $ and $v'_1$ belong to the same horizontal edge interval of $\P$.

If $v'_1$ and $v_1$ are in the same horizontal interval, 
then by applying Lemma \ref{reduction} to the vertices $v_1,v'_1,v'_2$,
we obtain that $f$ is redundant, a contradiction.
By using the same argument, we see that $v_2 \notin \I_e$

Suppose that $v'_1$ and $v_1$ are in the same vertical interval.
Assume that the $v$ is lower than $v'_1$. 
By using Lemma \ref{idk} (c), we observe that  $v_1,v_2,v'_1$ are
three corners of an inner interval. By applying Lemma \ref{reduction},
we see that $f$ is redundant.
Similarly, if $v'_1$ is lower than $v_1$ we obtain that $f$ is redundant.
Hence this case is not possible.

Finally, assume that $v_1$ and $v'_1$ are not in the same edge intervals.
If $v_1$ and $v_2$ belong to the same vertical edge interval,
Then by applying Lemma \ref{reduction} to the vertices
$v_2,v'_1,v'_2$ we are done.
Assume that the second coordinate of $v_1$ is less than that of $v'_1$. Let $g,h$ be the other corners of the inner interval defined by $v'_1$ and $v'_2$.
Assume that $v_1$,$v_2'$ and $g$ belong to the same vertical edge interval.
Then we have $x_{v'_1}x_{
v'_2}-x_gx_h \in J_\P$ and
\begin{eqnarray*}
f &=& f^+-f^-\\
&=&  f ^+ - x_{v'_1}x_{v'_2} \frac{f^-}{x_{v'_1}x_{v'_2}}\\
&=& f^+ - (x_{v'_1}x_{v'_2}-x_gx_h) \frac{f^-}{x_{v'_1}x_{v'_2}}-
x_gx_h\frac{f^-}{x_{v'_1}x_{v'_2}}.
\end{eqnarray*}

Let $f'  = f'^+-f'^-=f^+ - x_gx_h\frac{f^-}{x_{v'_1}x_{v'_2}}$
and let $V'_+$ and $V'_-$ be the vertices appearing in $f'^+$ and $f'^-$. Note that since $f$ and $x_{v'_1}x_{
v'_2}-x_gx_h$ are binomials belong to $J_\P$, $f' \in J_\P$. 

Then by applying Lemma \ref{reduction} to the vertices $v_1,v_2 \in V'_+$ and $g \in V'_-$, we obtain that $f'$ is redundant, which implies that $f$ is redundant.
Thus, the vertices appearing in $f$ does not belong to $\I_e$.
In other words, we have 
$f \in J_\P \cap K[x_{ij}\mid (i,j) \in V(\P) \setminus \Ic_e]$.

Let $\P'$ be the subpolyomino of $\P$ which consists of all cells of $\P$ having no vertices belonging to $\Ic_e$.
Then we have $I_{\P'} = I_\P \cap K[x_{ij}\mid (i,j) \in V(\P) \setminus \Ic_e]$.
We observe that $\P'$ is a simple polyomino.
Then, notice that $\alpha (v)$ for each $v \in \P \setminus \Ic_e$ 
is a monomial of degree 2 determined by the maximal horizontal and vertical intervals to which $v$ belongs.
Then, it is known from
 \cite[Theorem 2.2]{QSS} that  
 %$K[P']$ is isomorphic to $K[\alpha(v)|v \in \P']$.
  $I_{\P'}=I_\P \cap K[x_{ij}\mid (i,j) \in V(\P) \setminus \Ic_e] = J_\P \cap K[x_{ij}\mid (i,j) \in V(\P) \setminus \Ic_e]$.
Note that if $f$ is irredundant in $J_\P$, then it is also irredundant in $J_{\P} \cap K[x_{ij}\mid (i,j) \in V(\P) \setminus \Ic_e] $ since we have $J_{\P}\cap K[x_{ij}\mid (i,j) \in V(\P) \setminus \Ic_e] \subset J_{\P}$. 
We know that $J_{\P}\cap K[x_{ij}\mid (i,j) \in V(\P) \setminus \Ic_e]$ is generated by binomials of degree 2 since we have $I_{\P'}=J_{\P}\cap K[x_{ij}\mid (i,j) \in V(\P) \setminus \Ic_e]$ is generated by binomials of degree 2. This is a contradiction.
Hence we have the desired conclusion.

\end{proof}

\end{document}